\documentclass[11pt,reqno]{amsart}

\usepackage[T1]{fontenc}
\usepackage{lmodern}
\usepackage{amsmath,amssymb,mathtools,mathrsfs}
\usepackage{microtype}
\usepackage[colorlinks=true,linkcolor=blue,citecolor=blue,urlcolor=blue]{hyperref}
\hypersetup{pdftitle={Topological components of surface group representations into the unitary group},pdfauthor={Inkang Kim, Pierre Pansu, Xueyuan Wan}}
\providecommand{\doi}[1]{\href{https://doi.org/#1}{doi:~\nolinkurl{#1}}}

\newtheorem{theorem}{Theorem}[section]
\newtheorem{proposition}[theorem]{Proposition}
\newtheorem{lemma}[theorem]{Lemma}
\newtheorem{corollary}[theorem]{Corollary}
\theoremstyle{definition}

\theoremstyle{remark}
\newtheorem{remark}[theorem]{Remark}
\numberwithin{equation}{section}
\DeclareMathOperator{\U}{U}
\DeclareMathOperator{\SU}{SU}
\DeclareMathOperator{\Herm}{Herm}
\DeclareMathOperator{\tr}{tr}

\DeclareMathOperator{\Hom}{Hom}
\DeclareMathOperator{\sgn}{sgn}
\DeclareMathOperator{\sign}{sign}
\DeclareMathOperator{\Gr}{Gr}
\DeclareMathOperator{\diag}{diag}
\DeclareMathOperator{\im}{Im}

\newcommand{\C}{\mathbb C}
\newcommand{\Z}{\mathbb Z}
\newcommand{\E}{\mathcal E}
\newcommand{\HH}{\mathcal H}
\newcommand{\Srep}{\mathcal S}
\newcommand{\X}{\mathcal X}
\newcommand{\B}{\mathcal B}
\newcommand{\Id}{I_p}
\newcommand{\eps}{\varepsilon}

\title[Topological components of surface group representations]{Topological components of surface group representations into the unitary group}

\author{Xueyuan Wan}
\address{Mathematical Science Research Center, Chongqing University of Technology,
Chongqing 400054, China}
\email{xwan@cqut.edu.cn}

\date{}

\makeatletter
\@namedef{subjclassname@2020}{\textup{2020} Mathematics Subject Classification}
\makeatother
\thanks{
Xueyuan Wan is supported by the
National Key R\&D Program of China (Grant No. 2024YFA1013200) and the National Natural Science Foundation of China (Grant No. 12671100).}

\begin{document}

\begin{abstract}
We study representations of compact oriented surface groups into
$\mathrm U(p)$ with elliptic-unipotent boundary holonomies, meaning
that no boundary holonomy has eigenvalue $1$.
We prove that the signature of the associated flat Hermitian bundle
completely determines the connected component, and that every
component is path-connected. For genus $g\geq1$ and $n\geq1$ boundary
components, there are $np-1$ components; for $g=0$ and $n\geq2$,
there are $(n-2)p+1$. The disk case is empty, while the
closed-surface representation spaces are connected. The same
component classification holds after taking the quotient by
conjugation. Our proofs use the boundary rho invariant and
explicit matrix deformations. For a three-holed sphere, the
representation components have the homotopy types of complex
Grassmannians, and their conjugation quotients are contractible.
\end{abstract}

\maketitle

\section{Introduction}

The topology of surface group representation spaces is a central subject
in geometry and topology. These spaces connect flat bundles, moduli of
holomorphic bundles, and geometric structures on surfaces. Two basic
questions are to determine their connected components and to understand
the topology within each component. Classical results show how numerical
invariants and geometric structures can answer these questions. After
choosing a complex structure on a closed surface, the
Narasimhan--Seshadri correspondence relates irreducible unitary
representations, up to conjugacy, to stable holomorphic bundles of degree
zero \cite{NS1965}, and Atiyah--Bott developed a gauge-theoretic approach
to the topology of the corresponding moduli spaces through the
Yang--Mills functional \cite{AtiyahBott}. For a closed oriented surface
of genus $g\geq2$, Goldman proved that the Euler number completely
determines the connected component of a representation into
$\mathrm{PSL}(2,\mathbb R)$, giving $4g-3$ components; the two extremal
components, after taking the conjugation quotient, are copies of
Teichm\"uller space \cite[Theorem~B and Corollary~C]{Goldman1988}.

For surfaces with boundary, the choice of boundary conditions becomes
part of the topology of the problem. Goldman also classified components
for $\mathrm{PSL}(2,\mathbb R)$ with hyperbolic boundary holonomies
\cite[Theorem~D]{Goldman1988}. More recently, Kim--Wan \cite{KW2025}
used the signature formula to study representations into
$\mathrm{SL}(2,\mathbb R)$ and $\mathrm{PSL}(2,\mathbb R)$, obtaining
component counts for the loci defined by elliptic, hyperbolic, or
parabolic boundary conditions. Their results demonstrate the usefulness
of the signature and its boundary correction terms in questions about
deformations of surface group representations.

The compact unitary group provides a natural setting in which to examine
the role of the boundary. It is fundamental to the theory of flat
Hermitian bundles and also occurs as a factor of the maximal compact
subgroup of an indefinite unitary group. When the surface has nonempty
boundary, its fundamental group is free, and the unrestricted
$\U(p)$-representation space is a product of copies of $\U(p)$, hence
connected. Nevertheless, imposing an open condition on the boundary
holonomies can produce several components with different topology.
We study representations with \emph{elliptic-unipotent boundary
holonomies} in the sense of \cite[Definition~4.5]{KPW2022}.
Since every eigenvalue of a unitary matrix has modulus one, this
condition means precisely that no boundary holonomy has eigenvalue $1$.
We allow the remaining eigenvalues and the boundary conjugacy classes
to vary.

This boundary condition has both a topological and a spectral meaning.
The $1$-eigenspace of a boundary holonomy consists of the parallel
sections along that boundary circle. Its vanishing makes the local
system acyclic on the boundary and gives a nondegenerate twisted
intersection form on the whole first cohomology, of constant dimension.
At the same time, excluding $1$ permits a continuous choice of
eigenangles in $(0,2\pi)$ and removes the jumps of the boundary
rho invariant at that spectral cut. The signature is defined more
generally, but this open locus is a natural setting in which it becomes
a deformation invariant even as the boundary conjugacy classes vary.

The possible values of the signature are already known.  Kim--Pansu--Wan \cite{KPW2026}
determine, among other cases, all signatures of representations into
$\U(p)$, allowing boundary holonomies with eigenvalue $1$.
For genus zero and $n\geq2$ boundary components, the range is
$[-p(n-2),p(n-2)]\cap\mathbb Z$; for positive genus and $n\geq1$,
it is $([2-np,np-2]\cap\mathbb Z)\cup\{0\}$; see
\cite[Theorem~2 and Section~10.2]{KPW2026}.
Their recent preprint \cite[Theorem~1.1]{KPWUnitary2026} completes the
determination of possible signatures for general $\U(p,q)$, including
the previously unresolved unbalanced case in positive genus with
nonempty boundary. These numerical results motivate a more precise
topological question: for the compact unitary group, does a signature
value determine a single connected component on the locus where the
boundary holonomies are elliptic-unipotent, and what is the topology of that component?
The present paper answers the component question and gives explicit
deformations and homotopy descriptions in the matrix spaces underlying
the genus-zero case.

\subsection{The signature viewpoint}

Our starting point is the signature formula of Kim--Pansu--Wan
\cite[Theorem~1]{KPW2022}. For a representation
$\phi:\pi_1(\Sigma)\to\U(p,q)$ of a surface with boundary, it reads
\[
 \sign(\phi)=-2T(\Sigma,\phi)
       +\sum_{c\subset\partial\Sigma}\rho\bigl(\phi(c)\bigr),
\]
where $T$ is the Toledo invariant and $\rho$ is a conjugation-invariant
function of the boundary holonomy. For representations into the compact
group $\U(p)$, the Toledo term vanishes
\cite[Section~6.2]{KPW2022}. Thus the signature is determined entirely
by the boundary rho invariant, which is the numerical tool used
throughout this paper.

Let $\Sigma_{g,n}$ be a compact, connected, oriented surface of genus
$g$ with $n\geq0$ boundary components, and orient its boundary loops
by the induced boundary orientations. We use the presentation
\begin{equation}\label{eq:presentation}
 \pi_1(\Sigma_{g,n})=
 \left\langle a_1,b_1,\ldots,a_g,b_g,c_1,\ldots,c_n
 \;\middle|\;
 \prod_{r=1}^g[a_r,b_r]c_1\cdots c_n=1\right\rangle,
\end{equation}
where $[a,b]=aba^{-1}b^{-1}$ and the boundary generators are omitted
when $n=0$. Let
\[
 \E_p=\mathrm{EU}(p)
 =\{A\in\U(p):\det(A-\Id)\ne0\}
\]
be the elliptic-unipotent locus in $\U(p)$.
Every element of $\U(p)$ is semisimple, so its unipotent part
in the Jordan decomposition is the identity; the terminology
is inherited from the general indefinite unitary setting \cite[Definition~4.5]{KPW2022}.
Define
\[
 \Srep_{g,n}(p)=
 \{\phi\in\Hom(\pi_1(\Sigma_{g,n}),\U(p)):
       \phi(c_j)\in\E_p\text{ for every }j\}.
\]
For $n=0$, this condition is vacuous, so $\Srep_{g,0}(p)$ is the
full representation space. We first study representations themselves
and then pass to the quotient by simultaneous conjugation.

For $A\in\E_p$, write its eigenvalues, with multiplicities, as
$e^{i\theta_1},\ldots,e^{i\theta_p}$, where $0<\theta_\ell<2\pi$.
In the compact unitary normalization of \cite{KPW2022},
\begin{equation}\label{eq:intro-rho}
 \rho(A)=p-\frac1\pi\sum_{\ell=1}^p\theta_\ell.
\end{equation}
This is a continuous real-valued function on $\E_p$.
Define
\begin{equation}\label{eq:R-k-intro}
 \boldsymbol{\rho}(\phi)=\sum_{j=1}^n\rho\bigl(\phi(c_j)\bigr),
 \qquad
 k(\phi)=\frac{np-\boldsymbol{\rho}(\phi)}2.
\end{equation}
Equivalently, if $\phi(c_j)=e^{i\Theta_j}$ with
$0<\Theta_j<2\pi\Id$, then
$k(\phi)=(2\pi)^{-1}\sum_j\tr\Theta_j$.
Taking determinants in \eqref{eq:presentation} shows that $k$ is an
integer. It is continuous and therefore constant on each connected
component. Empty sums are understood to be zero, so $ \boldsymbol{\rho}=k=0$ when
$n=0$.

The signature formula specializes to
\begin{equation}\label{eq:signature-intro}
 \sign(\phi)=\boldsymbol{\rho}(\phi)=np-2k(\phi).
\end{equation}
For closed surfaces the same equality holds with all three terms zero.
In particular, when $n\geq1$, the boundary condition forces
$\sign(\phi)\equiv np\pmod2$. This explains why the signature values
on our open locus form a parity subset of the unrestricted ranges
described above. Knowing these values separates different levels;
the main topological issue is to prove that each nonempty level is
connected.

\subsection{Main results}

\begin{theorem}\label{thm:main}
Let $p\geq1$ and $g,n\geq0$. The nonempty fibers of $k$ are precisely
the connected components of $\Srep_{g,n}(p)$, and each is path connected.
The possible values and the number of components are
\[
\begin{array}{c|c|c}
 \text{surface}&\text{values of }k&\#\pi_0(\Srep_{g,n}(p))\\ \hline
 n=0&\{0\}&1\\
 g=0,\ n=1&\varnothing&0\\
 g=0,\ n\geq2&p,p+1,\ldots,(n-1)p&(n-2)p+1\\
 g\geq1,\ n\geq1&1,2,\ldots,np-1&np-1.
\end{array}
\]
The set $\{1,\ldots,np-1\}$ is understood to be empty when $np=1$.
Equivalently, either $\boldsymbol{\rho}$ or the signature is a complete invariant of
connected components. The same component classification and path
connectedness hold for the quotient $\Srep_{g,n}(p)/\U(p)$ with its
quotient topology.
\end{theorem}

The closed-surface statement is classical: the sphere gives a single
point, while connectedness in positive genus follows from
\cite[Theorem~1]{HL2005}. The representation spaces are real algebraic
and hence locally path connected, so they are also path connected.
This case is included to give a uniform description for all compact
oriented surfaces. For nonempty boundary, the theorem
says that two representations can be joined through representations
satisfying the boundary condition if and only if their signatures agree.

Once a handle is present, the number of components is independent of
the genus. This reflects the fact that every determinant-compatible
boundary tuple can be realized by handle holonomies, and the space of
such realizations is connected. In genus zero, the product relation
imposes additional restrictions. For fixed $n\geq2$, adding the first
handle allows exactly $2p-2$ further values of $k$. In rank one the
number is $n-1$ for every genus when $n\geq1$; in particular, the
one-boundary space is empty. More generally, a surface of positive
genus with one boundary component has $p-1$ components.

The matrix statement underlying the genus-zero proof also describes
the topology of the two-factor spaces. Here $m$ denotes the number of
matrix factors, to distinguish it from the number $n$ of boundary
components.

\begin{theorem}\label{thm:products-intro}
For $m\geq1$, let
\[
 \X_{p,m}=\{(A_1,\ldots,A_m)\in\E_p^m:
                         A_1\cdots A_m\in\E_p\}.
\]
The function
\[
 \Delta_m(A_1,\ldots,A_m)
   =\sum_{j=1}^m\rho(A_j)-\rho(A_1\cdots A_m)
\]
has image
$\{(m-1)p-2K:K=0,1,\ldots,(m-1)p\}$, and each of its fibers is
path connected. Thus $\X_{p,m}$ has $(m-1)p+1$ connected components.
For $m=2$, the component with $\Delta_2=2r-p$ has the homotopy type
of $\Gr_r(\C^p)$, for $0\leq r\leq p$.
\end{theorem}

In genus zero, eliminating the last boundary matrix identifies
$\Srep_{0,n}(p)$ with $\X_{p,n-1}$ for $n\geq2$, and
$\rho(A^{-1})=-\rho(A)$ identifies $\boldsymbol{\rho}$ with $\Delta_{n-1}$.
Only the individual factors and their final product are constrained
in Theorem~\ref{thm:products-intro}. Intermediate products may acquire
eigenvalue $1$ during a deformation. Requiring every partial product
to avoid $1$ instead produces $(p+1)^{m-1}$ components. Our proof
describes these finer pieces and constructs paths joining precisely
those with the same total index.

We also obtain consequences for the intersection form and for
irreducible representations. When $n\geq1$ and the space is nonempty,
boundary acyclicity gives
$\dim H^1(\Sigma_{g,n};\mathcal E_\phi)=-p\chi(\Sigma_{g,n})$.
Together with \eqref{eq:signature-intro}, this yields
\[
 b_+(\phi)=p(g+n-1)-k(\phi),\qquad
 b_-(\phi)=p(g-1)+k(\phi).
\]
For $g=0$ and $n\geq3$, the two extreme components therefore have
positive or negative definite intersection forms. Moreover, when
$p\geq2$, $n\geq1$, and $2g+n-1\geq2$, the irreducible part of
each component is nonempty and path connected. Thus $k$ also labels
the components of the smooth irreducible quotient.  These results are developed in
Section~\ref{sec:applications}.

\subsection{Relation to earlier work}

The signature interpretation belongs to the theory of signatures of
local systems and spectral boundary corrections developed by Meyer,
Atiyah, and Atiyah--Patodi--Singer
\cite{MeyerLocal,MeyerSurface,Atiyah1987,APS1975I,APS1975II}.
Kim--Pansu--Wan \cite{KPW2022} express the signature in terms of
the Toledo invariant and an explicitly computable boundary
rho invariant. Their signature range theorems
\cite{KPW2026,KPWUnitary2026} concern which integers occur.
Here we study the topology of the corresponding levels on the open
unitary locus specified above. For $n\geq2$, the rank-one component
classification already appears in Kim--Wan
\cite[Proposition~2.2 and Corollary~2.3]{KW2025}, through
$\mathrm{SO}(2)\cong\U(1)$. Our signature is the complex Hermitian
signature; their real symplectic normalization is twice this value
on the corresponding rank-one unitary representations.

The positive-genus proof uses classical connectedness results for
compact groups. Alekseev--Malkin--Meinrenken's theorem on group-valued
moment maps implies connectedness of the fibers of the
product-of-commutators map for compact simply connected groups
\cite[Theorem~7.2]{AMM1998}; this consequence is stated explicitly
in \cite[Fact~3]{HL2003}. Ho--Liu also describe components with
prescribed boundary conjugacy classes \cite[Theorem~3]{HL2005}.
For $\U(p)$ in positive genus, their result gives connectedness
whenever the prescribed boundary determinants have product one.
We use this established input, explain the passage from $\SU(p)$
to $\U(p)$, and show how the boundary data can vary within a fixed
$k$-level without creating additional components.

In genus zero, the relation $C_1\cdots C_n=\Id$ is the multiplicative
eigenvalue problem. Agnihotri--Woodward and Belkale describe its
spectral constraints \cite{AW1998,Belkale2001}, and
Falbel--Wentworth formulate the total eigenangle index and the
associated affine inequalities
\cite[Theorems~2.1 and~2.2]{FW2006}. Their index is $k$ in our
normalization. On the locus where every boundary holonomy avoids
$1$, the admissible spectra of fixed index form a convex set, and
their index bounds become $p\leq k\leq(n-1)p$.
Combining this convexity with connectedness of fixed-spectrum fibers
gives an alternative proof of the genus-zero component classification.
The correspondence between unitary representations and polystable
parabolic bundles of parabolic degree zero \cite{MS1980} provides
the algebro-geometric background for this spectral theory.

Thus the component counts can also be obtained from classical
convexity and connectedness results. Our treatment places them in
a uniform rho invariant formulation and gives a direct matrix
proof in genus zero that does not require the multiplicative
eigenvalue inequalities. The explicit deformations explain how the
partial-product pieces meet and provide the Grassmannian homotopy
models and quotient contractions described above. The alternative
classical argument is given in Section~\ref{sec:classical}.

\subsection{Outline}

Section~\ref{sec:prelim} fixes the logarithmic and Cayley conventions.
Section~\ref{sec:pairs} treats two factors and identifies the homotopy
types of their components. Section~\ref{sec:products} proves the
general matrix theorem by constructing and joining Cayley-coordinate
pieces. Section~\ref{sec:surfaces} proves the surface classification.
Section~\ref{sec:applications} develops its consequences for
signatures, irreducible representations.
Section~\ref{sec:classical} explains a second route in genus zero.

\medskip
\noindent\textbf{Acknowledgments.}
Xueyuan Wan was supported by the National Key R\&D Program of China (Grant No.~2024YFA1013200) and the National Natural Science Foundation of China (Grant No.~12671100). The author used ChatGPT as an auxiliary tool for literature searches and for improving the language, grammar, and presentation of the manuscript. All AI-assisted content, including any suggested references or revisions, was reviewed, verified, and revised by the author. The author takes full responsibility for the accuracy, originality, and integrity of the manuscript.

\section{Logarithms, rho invariant, and Cayley coordinates}\label{sec:prelim}

\subsection{The logarithm with a cut at \texorpdfstring{$1$}{1}}

Let $\Herm(p)$ denote the real vector space of complex Hermitian $p\times p$ matrices. For Hermitian matrices, inequalities mean inequalities of quadratic forms. Set
\[
 \mathcal D_p=\{\Theta\in\Herm(p):0<\Theta<2\pi\Id\}.
\]
This is an open convex set. The scalar map
\(
(0,2\pi)\to S^1\setminus\{1\},
\) \(
\theta\to e^{i\theta},
\)
is a smooth bijection. Applying this fact to the eigenvalues gives a
diffeomorphism
\[
\mathcal D_p\longrightarrow\mathcal E_p,
\qquad
\Theta\longmapsto e^{i\Theta}.
\]
Indeed, if \(A\in\mathcal E_p\), then the spectral theorem gives
\(
A
 =U\operatorname{diag}
   \bigl(e^{i\theta_1},\ldots,e^{i\theta_p}\bigr)U^*,
\) \(
0<\theta_j<2\pi,
\)
where each \(\theta_j\) is uniquely determined. We define
\(
\Theta(A)
 =U\operatorname{diag}(\theta_1,\ldots,\theta_p)U^*.
\)
This definition is independent of the chosen diagonalization, and
\(\Theta(A)\) is the unique matrix in \(\mathcal D_p\) satisfying
\(e^{i\Theta(A)}=A\). Moreover,
\(
\Theta(A)=-i\mathrm{Log}_{(0,2\pi)}(A),
\)
where \(\mathrm{Log}_{(0,2\pi)}\) denotes the branch of the logarithm whose
imaginary part lies in \((0,2\pi)\). Since the spectrum of every
\(A\in\mathcal E_p\) avoids the corresponding branch cut, the
holomorphic functional calculus shows that \(A\mapsto\Theta(A)\)
is smooth.

Consequently,
\begin{equation}\label{eq:rho-log}
 \rho(A)=p-\pi^{-1}\tr\Theta(A),
 \qquad \det A=\exp(i\tr\Theta(A)).
\end{equation}
The function $\rho$ is smooth and invariant under conjugation. The spectral theorem also gives
\begin{equation}\label{eq:rho-inverse-sum}
 \Theta(A^{-1})=2\pi\Id-\Theta(A),\,
 \rho(A^{-1})=-\rho(A),\,
 \rho(A\oplus B)=\rho(A)+\rho(B).
\end{equation}
In the direct-sum identity, \(\rho\) is computed in rank \(p+q\) on
the left-hand side and in ranks \(p\) and \(q\), respectively, on the
right-hand side. Notice that $-p<\rho(A)<p$.

\begin{lemma}\label{lem:discrete}
The function $\Delta_m$ of Theorem~\ref{thm:products-intro} is continuous and takes values in $(m-1)p+2\Z$. For a surface representation in $\Srep_{g,n}(p)$, the number $k$ in \eqref{eq:R-k-intro} is an integer. If $n\geq1$, then
$0<k<np$; if $n=0$, then $k=0$.
\end{lemma}
\begin{proof}
Put $P=A_1\cdots A_m$. Multiplicativity of the determinant implies
\[
 \exp\left(i[\sum_{j=1}^m\tr\Theta(A_j)-\tr\Theta(P)]\right)=1.
\]
The expression in square brackets is therefore $2\pi K$ for an integer $K$. Substituting into \eqref{eq:rho-log} gives
\begin{equation}\label{eq:Delta-K}
 \Delta_m=(m-1)p-2K.
\end{equation}
Continuity follows from that of $\Theta$; thus both $\Delta_m$ and $K$ are locally constant.

For a surface representation, every commutator has determinant one, so \eqref{eq:presentation} yields $\prod_j\det\phi(c_j)=1$. Hence $\sum_j\tr\Theta(\phi(c_j))\in2\pi\Z$. Each summand lies strictly between $0$ and $2\pi p$, proving the asserted bounds.
\end{proof}

\subsection{Two Cayley identities}

Define
\begin{equation}\label{eq:Cayley}
 h(A)=i(A+\Id)(A-\Id)^{-1},\quad
 C(H)=(H+i\Id)(H-i\Id)^{-1}.
\end{equation}
The scalar identity
\[
 i\frac{e^{i\theta}+1}{e^{i\theta}-1}=\cot(\theta/2)
\]
and the spectral theorem show that $h:\E_p\to\Herm(p)$ and $C:\Herm(p)\to\E_p$ are inverse diffeomorphisms. In particular $\E_p$ is contractible. Since
$\theta=\pi-2\arctan(\cot(\theta/2))$ on $(0,2\pi)$, one also has
\begin{equation}\label{eq:rho-arctan}
 \rho(C(H))=\frac2\pi\tr(\arctan H).
\end{equation}
Here \(\arctan:\mathbb R\to(-\pi/2,\pi/2)\) denotes the principal
branch. For \(0<\theta<2\pi\), we have
\(
\cot(\theta/2)
=\tan(\frac{\pi}{2}-\frac{\theta}{2}),
\)
and \(\frac{\pi}{2}-\frac{\theta}{2}\in(-\pi/2,\pi/2)\). Hence
\(
\theta
=\pi-2\arctan(\cot\frac{\theta}{2}).
\)
Now let \(\lambda_1,\ldots,\lambda_p\) be the eigenvalues of
\(H\). If \(e^{i\theta_j}\), with \(0<\theta_j<2\pi\), is the
corresponding eigenvalue of \(C(H)\), then
\(
\lambda_j=\cot(\theta_j/2),
\) \(
\theta_j=\pi-2\arctan(\lambda_j).
\)
Therefore,
\[
\begin{aligned}
\rho(C(H))
 &=p-\frac{1}{\pi}\sum_{j=1}^p\theta_j=\frac{2}{\pi}\sum_{j=1}^p\arctan(\lambda_j)
 =\frac{2}{\pi}\tr(\arctan H).
\end{aligned}
\]

\begin{lemma}\label{lem:cayley-identities}
For $H,K\in\Herm(p)$, the following identities hold:
\begin{equation}\label{eq:cayley-product}
   (H-i\Id)(C(H)C(K)-\Id)(K-i\Id)=2i(H+K),
\end{equation}
\begin{equation} \label{eq:cayley-difference}
    C(H)-C(K)=2i(H-i\Id)^{-1}(K-H)(K-i\Id)^{-1}.
\end{equation}
Consequently, $C(H)C(K)-\Id$ is invertible exactly when $H+K$ is invertible, and $C(H)-C(K)$ is invertible exactly when $H-K$ is invertible.
\end{lemma}
\begin{proof}
The matrices $H-i\Id$ and $K-i\Id$ are invertible. Moreover $H-i\Id$ commutes with $C(H)$, because both are functions of $H$. Thus the left side of \eqref{eq:cayley-product} equals
\[
 (H+i\Id)(K+i\Id)-(H-i\Id)(K-i\Id)=2i(H+K).
\]
This computation does not require $H$ and $K$ to commute. For the second identity, use $C(H)=\Id+2i(H-i\Id)^{-1}$ and the resolvent identity
\[
 (H-i\Id)^{-1}-(K-i\Id)^{-1}
 =(H-i\Id)^{-1}(K-H)(K-i\Id)^{-1}.
\]
The invertibility assertions follow by multiplying by the invertible outside factors.
\end{proof}

\section{Two factors and the pair of pants}\label{sec:pairs}

Let \(\Herm(p)^\times\) denote the space of invertible Hermitian
\(p\times p\) matrices. For an invertible Hermitian matrix $S\in \Herm(p)^\times$, let $n_+(S)$ and $n_-(S)$ be its numbers of positive and negative eigenvalues. Write
$\sign S=n_+(S)-n_-(S)$, and set
\[
 \HH_r=\{S\in\Herm(p)^\times:n_+(S)=r\}.
\]

\begin{lemma}\label{lem:Herm-components}
The sets $\HH_r$, $0\leq r\leq p$, are the path components of the invertible Hermitian matrices. Moreover, \(\HH_r\) strongly deformation retracts onto the subspace
\(
\{2P_V-I_p:V\in\Gr_r(\C^p)\},
\)
where \(P_V\) denotes the orthogonal projection onto \(V\). This
subspace is naturally identified with the Grassmannian
\(\Gr_r(\C^p)\). Consequently, \(\HH_r\) has the homotopy type of
\(\Gr_r(\C^p)\).\end{lemma}
\begin{proof}
We first observe that the number of positive
eigenvalues is constant along every path in \(\Herm(p)^\times\).
Indeed, the eigenvalues of a Hermitian matrix vary continuously with
the matrix. Thus, for an eigenvalue to change from positive to
negative, it would have to pass through \(0\). At that moment the
matrix would no longer be invertible. Consequently, the integer $n_+(S)$
 is locally constant on
\(\Herm(p)^\times\). It follows that every path component of
\(\Herm(p)^\times\) is contained in one of the sets \(\HH_r\).

We next show that each \(\HH_r\) is path-connected and determine its
homotopy type. For \(S\in\HH_r\), define the matrix sign of \(S\) by
\(
\sgn(S)=S(S^2)^{-1/2}.
\)
Since \(S\) is Hermitian and invertible, \(S^2\) is positive definite,
so its positive inverse square root \((S^2)^{-1/2}\) is well defined.
If
\(
S=U\diag(\lambda_1,\ldots,\lambda_p)U^*
\)
is a spectral decomposition of \(S\), then
\[
\sgn(S)
=
U\diag\bigl(
\sgn(\lambda_1),\ldots,\sgn(\lambda_p)
\bigr)U^*.
\]
Thus \(\sgn(S)\) is obtained from \(S\) by replacing every positive
eigenvalue by \(1\) and every negative eigenvalue by \(-1\), without
changing the corresponding eigenspaces. In particular,
\(\sgn(S)^*=\sgn(S)\), \(\sgn(S)^2=I_p,
\)
and \(\sgn(S)\) has exactly \(r\) positive eigenvalues.

Consider the homotopy
\[
F_t(S)=(1-t)S+t\sgn(S),
\qquad 0\leq t\leq1.
\]
The map \(S\mapsto\sgn(S)\) is continuous, indeed smooth, on the
space of invertible Hermitian matrices, so \(F_t(S)\) depends
continuously on both \(S\) and \(t\). In the above eigenbasis of \(S\),
the eigenvalues of \(F_t(S)\) are
\(
(1-t)\lambda_j+t\sgn(\lambda_j).
\)
If \(\lambda_j>0\), then this number remains positive for every
\(t\in[0,1]\); if \(\lambda_j<0\), then it remains negative.
Therefore no eigenvalue of \(F_t(S)\) passes through \(0\), and
\(
F_t(S)\in\HH_r
\) for every $0\leq t\leq1$.
Moreover,
\(
F_0(S)=S,
\) \(
F_1(S)=\sgn(S).
\)

Let
\(
\mathcal J_r
=
\{J\in\HH_r:J^2=I_p\}
\)
be the space of Hermitian involutions having \(r\) positive
eigenvalues. If \(J\in\mathcal J_r\), then \(\sgn(J)=J\), and hence
$
F_t(J)=J
$ for every $t\in[0,1]$. 
Thus \(F\) is a strong deformation retraction of \(\HH_r\) onto
\(\mathcal J_r\).

It remains to identify \(\mathcal J_r\). Every \(J\in\mathcal J_r\)
has only the eigenvalues \(1\) and \(-1\). Let \(V\subset\C^p\) be
its \(+1\)-eigenspace. Then \(\dim_{\C}V=r\), and \(V^\perp\) is the
\(-1\)-eigenspace. Therefore
\(
J=P_V-P_{V^\perp}=2P_V-I_p,
\)
where \(P_V\) is the orthogonal projection onto \(V\). Conversely,
for every \(V\in\Gr_r(\C^p)\), the matrix \(2P_V-I_p\) belongs to
\(\mathcal J_r\). Hence
\[
\mathcal J_r
=
\{2P_V-I_p:V\in\Gr_r(\C^p)\}
\cong
\Gr_r(\C^p)
=
\U(p)/\bigl(\U(r)\times\U(p-r)\bigr).
\]

The Grassmannian \(\Gr_r(\C^p)\) is path connected because
\(\U(p)\) is path connected and acts transitively on it. The strong
deformation retraction therefore implies that \(\HH_r\) is path
connected. Since the number of positive eigenvalues cannot change
along a path of invertible Hermitian matrices, two different sets
\(\HH_r\) and \(\HH_s\), with \(r\neq s\), cannot lie in the same
path component. Consequently, the sets
\(
\HH_0,\HH_1,\ldots,\HH_p
\)
are precisely the path components of \(\Herm(p)^\times\), and each
\(\HH_r\) strongly deformation retracts onto
\(\Gr_r(\C^p)\).
\end{proof}

\begin{proposition}\label{prop:pair}
The map
\[
\Phi: \X_{p,2}\longrightarrow\Herm(p)\times\Herm(p)^{\times},\quad
\Phi (A,B)= (h(A),h(A)+h(B))
\]
is a diffeomorphism. Moreover,
\begin{equation}\label{eq:pair-signature}
 \rho(A)+\rho(B)-\rho(AB)=\sign(h(A)+h(B)).
\end{equation}
Thus $\X_{p,2}$ has $p+1$ components. The component with $\Delta_2=2r-p$ is diffeomorphic to $\Herm(p)\times\HH_r$ and has the homotopy type of $\Gr_r(\C^p)$.
\end{proposition}

\begin{proof}
We first describe the displayed map and its inverse more explicitly.
Write
$
H=h(A)$, $K=h(B).
$
Since \(h:\mathcal E_p\to\Herm(p)\) is a diffeomorphism with inverse
\(C\), every pair \((A,B)\in\mathcal E_p^2\) can be written uniquely
as
$
A=C(H),$ $B=C(K)
$
for some \(H,K\in\Herm(p)\).

By \eqref{eq:cayley-product}, the matrix
$
C(H)C(K)-I_p
$
is invertible if and only if \(H+K\) is invertible. Since
\(A=C(H)\) and \(B=C(K)\), this says precisely that $AB\in\mathcal E_p$ if and only if $H+K\in\Herm(p)^\times$.
One can check directly that
\[
\Phi^{-1}(H,S)=\bigl(C(H),C(S-H)\bigr).
\]
Since \(h\), \(C\), and matrix addition are smooth,
both \(\Phi\) and \(\Phi^{-1}\) are smooth. Therefore \(\Phi\) is a
diffeomorphism.

The first factor \(\Herm(p)\) is a real vector space and is therefore
path connected. By Lemma~\ref{lem:Herm-components}, the path
components of \(\Herm(p)^\times\) are
$
\HH_r
$, $0\leq r\leq p$.
It follows that the path components of \(\X_{p,2}\) are precisely
\(
\Phi^{-1}\bigl(\Herm(p)\times\HH_r\bigr),\, 0\leq r\leq p.
\)

It remains to determine the value of
\[
\Delta_2(A,B)=\rho(A)+\rho(B)-\rho(AB)
\]
on the component corresponding to \(\HH_r\). By
Lemma~\ref{lem:discrete}, \(\Delta_2\) is locally constant. Since
each of the above components is path-connected, \(\Delta_2\) is
constant on each component. We may therefore compute its value at
one convenient point.

Fix \(r\) and let
\(
J_r=\diag(I_r,-I_{p-r}).
\)
Choose
\(
h(A)=h(B)=J_r.
\)
Equivalently, in the coordinates \((H,S)\) introduced above, this
means
\(
H=J_r,S=h(A)+h(B)=2J_r.
\)
The matrix \(2J_r\) has \(r\) positive and \(p-r\) negative
eigenvalues, so it belongs to \(\HH_r\). Since
\(
C(1)=i,\,C(-1)=-i,
\)
the corresponding matrices are
\(
A=B=C(J_r)
 =\diag(iI_r,-iI_{p-r}).
\)
Consequently,
\(
AB=-I_p.
\)

Recall that an eigenvalue \(e^{i\theta}\), with
\(0<\theta<2\pi\), contributes
\(
1-\frac{\theta}{\pi}
\)
to \(\rho\). Hence \(i=e^{i\pi/2}\) contributes \(1/2\), whereas
\(-i=e^{3i\pi/2}\) contributes \(-1/2\). It follows that
\(
\rho(A)=\rho(B)
 =\frac{r}{2}-\frac{p-r}{2}
 =r-\frac{p}{2}.
\)
Every eigenvalue of \(-I_p\) has argument \(\pi\), so
\(
\rho(-I_p)=0.
\)
Therefore
\[
\Delta_2(A,B)
 =\rho(A)+\rho(B)-\rho(AB)
 =2\left(r-\frac{p}{2}\right)
 =2r-p.
\]

On the other hand, every matrix \(S\in\HH_r\) has \(r\) positive
and \(p-r\) negative eigenvalues. Its signature is therefore
\(
\sign(S)=r-(p-r)=2r-p.
\)
Thus \(\Delta_2\) and
\(\sign\bigl(h(A)+h(B)\bigr)\) have the same value on the component
corresponding to \(\HH_r\). We conclude that
\[
\rho(A)+\rho(B)-\rho(AB)
 =\sign\bigl(h(A)+h(B)\bigr)
\]
for every \((A,B)\in \X_{p,2}\), proving
\eqref{eq:pair-signature}.

Finally, the component corresponding to \(\HH_r\) is diffeomorphic
to
\(
\Herm(p)\times\HH_r.
\)
The vector space \(\Herm(p)\) contracts to \(0\), while
Lemma~\ref{lem:Herm-components} shows that \(\HH_r\) strongly
deformation retracts onto \(\Gr_r(\C^p)\). Hence this component has
the homotopy type of \(\Gr_r(\C^p)\).
\end{proof}

\begin{remark}\label{rem:explicit-pair-path}
The proposition also supplies a concrete path between any two pairs with the same $\Delta_2$. Put $H_\nu=h(A^\nu)$ and $S_\nu=h(A^\nu)+h(B^\nu)$ for $\nu=0,1$. Choose spectral decompositions $S_\nu=U_\nu D_\nu U_\nu^*$ with the positive diagonal entries first. Choose a Hermitian $T$ with $e^{iT}=U_0^*U_1$, and set
\[
 \begin{split}
 S(t)&=U_0e^{itT}((1-t)D_0+tD_1)e^{-itT}U_0^*,\\
 H(t)&=(1-t)H_0+tH_1,\\
 A(t)&=C(H(t)),\qquad B(t)=C(S(t)-H(t)).
 \end{split}
\]
The matching signs of the diagonal entries guarantee that $S(t)$ is invertible. The Cayley identity gives $A(t),B(t),A(t)B(t)\in\E_p$, with the required endpoints. A Hermitian logarithm $T$ exists by the spectral theorem for unitary matrices; no continuous choice as the endpoints vary is needed.
\end{remark}

For the pair of pants, the third holonomy is $(AB)^{-1}$. Hence $\boldsymbol{\rho}=\Delta_2$, and its component with $k=2p-r$ has homotopy type $\Gr_r(\C^p)$. In particular, in rank two the three components have $\boldsymbol{\rho}=-2,0,2$; the middle component has homotopy type $\mathbb{CP}^1$, while the other two are contractible.

\begin{corollary}\label{cor:pants-quotient}
Each component of $\Srep_{0,3}(p)/\U(p)$ is contractible.
\end{corollary}

\begin{proof}
Fix the component of \(\X_{p,2}\) corresponding to
\(
\HH_r
\)
We construct a strong deformation retraction of this component onto
a single \(\U(p)\)-orbit, and we do so equivariantly with respect to
simultaneous conjugation.

Let \((A,B)\) belong to this component and write
$H=h(A)$, $K=h(B)$, and $S=H+K$.
By \eqref{eq:cayley-product}, the condition \(AB\in\E_p\) is
equivalent to the invertibility of \(S=H+K\). Moreover, the component
under consideration is characterized by
$
n_+(S)=r.
$

Define
$
J=\sgn(S)=S(S^2)^{-1/2}.
$
If
$
S=U\diag(\lambda_1,\ldots,\lambda_p)U^*
$
is a spectral decomposition, then
$
J
 =U\diag\bigl(
 \sgn(\lambda_1),\ldots,\sgn(\lambda_p)
 \bigr)U^*.
$
We now deform \(H\) and \(K\) simultaneously toward \(J\). Set
\[
H_t=(1-t)H+tJ,\qquad
K_t=(1-t)K+tJ,
\qquad 0\leq t\leq1,
\]
and define
\[
\mathcal R_t(A,B)
 =\bigl(C(H_t),C(K_t)\bigr).
\]
Since \(H_t\) and \(K_t\) are Hermitian, both \(C(H_t)\) and
\(C(K_t)\) belong to \(\E_p\). It remains to check that their product
also belongs to \(\E_p\).

We have
$
H_t+K_t=(1-t)S+2tJ.
$
Because \(J=\sgn(S)\) is a function of \(S\), the matrices \(S\) and
\(J\) have the same eigenspaces.Therefore \(H_t+K_t\) remains invertible and has exactly
\(r\) positive eigenvalues throughout the deformation.
It now follows from \eqref{eq:cayley-product} that
\(
C(H_t)C(K_t)\in\E_p
\)
for every \(t\). Thus \(\mathcal R_t(A,B)\) remains in the same
component of \(\X_{p,2}\).

At \(t=0\), we recover the original pair:
$
\mathcal R_0(A,B)
 =\bigl(C(H),C(K)\bigr)=(A,B).
$
At \(t=1\), we have \(H_1=K_1=J\), and hence
$
\mathcal R_1(A,B)=\bigl(C(J),C(J)\bigr)=(iJ,iJ).
$
Let
\[
\mathcal O_r
 =
 \bigl\{(iJ,iJ):
 J=J^*=J^{-1},\ n_+(J)=r\bigr\}.
\]
The preceding construction shows that the deformation ends in
\(\mathcal O_r\). Moreover, it fixes every point of \(\mathcal O_r\).
Indeed, if \((A,B)=(iJ_0,iJ_0)\in\mathcal O_r\), then
$
H=h(A)=J_0,$ $K=h(B)=J_0,
$
and hence
$
\sgn(H+K)=\sgn(2J_0)=J_0.
$
It follows that \(H_t=K_t=J_0\) for every \(t\), so
$
\mathcal R_t(iJ_0,iJ_0)=(iJ_0,iJ_0).
$
Therefore \(\mathcal R_t\) is a strong deformation retraction onto
\(\mathcal O_r\).

We next verify that the deformation is compatible with simultaneous
unitary conjugation. This simply means that changing the orthonormal
basis before applying the deformation gives the same result as changing
the basis afterward.
Indeed, functional calculus is compatible with unitary conjugation:
for every \(U\in\U(p)\),
$
h(UAU^*)=Uh(A)U^*$, $C(UHU^*)=UC(H)U^*$
and
$
\sgn(USU^*)=U\sgn(S)U^*.
$
Thus, if
$
\mathcal R_t(A,B)=(A_t,B_t),
$
then
\[
\mathcal R_t(UAU^*,UBU^*)
=
(UA_tU^*,UB_tU^*).
\]
Hence \(\mathcal R_t\) is equivariant under simultaneous conjugation.
In particular, it induces a well-defined deformation on the quotient
\(\X_{p,2}/\U(p)\).

Finally, every Hermitian involution with \(r\) positive eigenvalues
is unitarily conjugate to
$
J_r=\diag(I_r,-I_{p-r}).
$
Hence
$
\mathcal O_r
 =\U(p)\cdot(iJ_r,iJ_r)
$ (action by conjugation)
is a single \(\U(p)\)-orbit, naturally identified with
$
\U(p)/\bigl(\U(r)\times\U(p-r)\bigr)
 =\Gr_r(\C^p).
$
Since the deformation is equivariant, it descends to the quotient.
The whole quotient component is thereby deformed onto
\(\mathcal O_r/\U(p)\), which consists of a single point. Therefore
each component of \(\X_{p,2}/\U(p)\) is contractible.

For the pair of pants, the identification
\(
\X_{p,2}\to S_{0,3}(p)
\) with \(
(A,B)\mapsto\bigl(A,B,(AB)^{-1}\bigr),
\)
is equivariant under simultaneous conjugation. Hence the same
conclusion holds for every component of
\(S_{0,3}(p)/\U(p)\).
\end{proof}

\section{Arbitrary products: an explicit deformation argument}\label{sec:products}

We now prove Theorem~\ref{thm:products-intro}. The case $m=1$ is immediate: $\X_{p,1}=\E_p$ is contractible and $\Delta_1=0$. Assume $m\geq2$.

\subsection{The locus with admissible partial products}

For a tuple in $\X_{p,m}$, let $P_j=A_1\cdots A_j$ and consider the open subset
\[
 \mathcal R_{p,m}=\{(A_1,\ldots,A_m)\in\X_{p,m}:
                         P_j\in\E_p\text{ for }1\leq j\leq m\}.
\]
The extra conditions only concern $P_2,\ldots,P_{m-1}$. On this subset put
\begin{equation}\label{eq:prefix-coordinates}
 M_j=h(P_j),\qquad D_j=M_{j-1}-M_j\quad(2\leq j\leq m).
\end{equation}

\begin{lemma}\label{lem:prefix-chart}
The coordinates $(M_1,D_2,\ldots,D_m)$ define a diffeomorphism
\[
 \Psi_m:\mathcal R_{p,m}\longrightarrow \Herm(p)\times(\Herm(p)^{\times})^{m-1}.
\]
In particular, its path components are indexed by vectors
$\mathbf q=(q_2,\ldots,q_m)\in\{0,\ldots,p\}^{m-1}$, where $q_j=n_-(D_j)$.
\end{lemma}
\begin{proof}
Since $A_j=P_{j-1}^{-1}P_j$ for $j\geq2$,
\[
 A_j-\Id=P_{j-1}^{-1}(P_j-P_{j-1}).
\]
Equation~\eqref{eq:cayley-difference} therefore shows that $A_j-\Id$ is invertible if and only if $M_{j-1}-M_j=D_j$ is invertible. 
Hence, $A_j\in \mathcal{E}_p$ if and only if $D_j\in \Herm(p)^{\times}.$
Conversely, given arbitrary $M_1\in\Herm(p)$ and invertible Hermitian matrices $D_2,\ldots,D_m$, define
\[
 M_j=M_1-\sum_{r=2}^jD_r,\quad
 P_j=C(M_j),\quad A_1=P_1,\quad A_j=P_{j-1}^{-1}P_j.
\]
These formulas produce a tuple in $\mathcal R_{p,m}$ and invert \eqref{eq:prefix-coordinates}. 
Both directions
are smooth, since they use only the smooth maps \(h\) and \(C\),
matrix addition, multiplication, and inversion. Hence $\Psi_m$
is a diffeomorphism.

Finally, \(\Herm(p)\) is a real vector space and is therefore
path connected. By Lemma~\ref{lem:Herm-components}, a path
component of \(\Herm(p)^\times\) is determined by the number of
positive, or equivalently negative, eigenvalues.

Following our convention, put \(q_j=n_-(D_j)\). Since \(D_j\)
is invertible, it then has \(p-q_j\) positive eigenvalues.
Thus, for a fixed vector
\(\mathbf q=(q_2,\ldots,q_m)\in\{0,\ldots,p\}^{m-1}\),
the corresponding coordinate subset is
\(
\Herm(p)\times
\HH_{p-q_2}\times\cdots\times\HH_{p-q_m}.
\)
Every factor is path connected, so this product is path connected.
Different choices of \(\mathbf q\) cannot be joined by a path,
because the number of negative eigenvalues of an invertible
Hermitian matrix cannot change along such a path.
Therefore the path components of \(\mathcal R_{p,m}\) are exactly
\[
\mathcal R_{\mathbf q}
 :=
\Psi_m^{-1}\!\left(
\Herm(p)\times
\HH_{p-q_2}\times\cdots\times\HH_{p-q_m}
\right),
\]
as claimed.
\end{proof}

\subsection{Diagonal representatives and the value of the invariant}

Fix $\eps=\pi/(2m)$. For a given vector $\mathbf q$, choose numbers
\[
 c_{j\ell}\in\{0,1\},\qquad
 \sum_{\ell=1}^p c_{j\ell}=q_j,
 \quad 2\leq j\leq m,\quad 1\leq\ell\leq p.
\]
Define a diagonal tuple by
\begin{equation}\label{eq:binary-representatives}
 A_1=-\Id,\quad
 A_j=\diag(e^{i\theta_{j1}},\ldots,e^{i\theta_{jp}}),\quad
 \theta_{j\ell}=\begin{cases}
 \eps&c_{j\ell}=0,\\
 2\pi-\eps&c_{j\ell}=1.
 \end{cases}
\end{equation}
Since \(c_{r\ell}\in\{0,1\}\), the definition of
\(\theta_{r\ell}\) gives
$
e^{i\theta_{r\ell}}
 =e^{i\eps(1-2c_{r\ell})}.
$
Thus each factor contributes an angle increment of either
\(\eps\) or \(-\eps\). Since \(A_1=-\Id\) and all the matrices
are diagonal, the \(\ell\)th diagonal entry of
\(P_j=A_1\cdots A_j\) is \(e^{i\varphi_{j\ell}}\), where
\begin{equation}\label{eq:prefix-angles}
\varphi_{j\ell}
 =\pi+\eps\sum_{r=2}^j(1-2c_{r\ell}).
\end{equation}
For \(j=1\), the sum is empty and \(\varphi_{1\ell}=\pi\).
Each summand is either \(1\) or \(-1\), so
\[
|\varphi_{j\ell}-\pi|
 \le (j-1)\eps
 \le (m-1)\eps
 <\frac{\pi}{2}.
\]
Hence \(\pi/2<\varphi_{j\ell}<3\pi/2\). In particular,
\(\varphi_{j\ell}\) already lies in the chosen argument interval
\((0,2\pi)\), so no adjustment by a multiple of \(2\pi\) is needed.
Every eigenvalue of \(P_j\) therefore lies in the open left
semicircle and is different from \(1\). Consequently,
\(P_j\in\E_p\) for every \(1\leq j\leq m\).

The $\ell$th entry of $D_j$ is
\[
 \cot(\varphi_{j-1,\ell}/2)-\cot(\varphi_{j\ell}/2).
\]
The function $\cot(x/2)$ is strictly decreasing on $(0,2\pi)$. Thus this entry is positive when $c_{j\ell}=0$ and negative when $c_{j\ell}=1$. The tuple lies in $\mathcal R_{\mathbf q}$.

We compute \(\Delta_m\) at the diagonal representative constructed
above. For each \(1\leq\ell\leq p\), let
$
d_\ell=\sum_{j=2}^m c_{j\ell}.
$
Thus \(d_\ell\) counts the factors whose \(\ell\)th diagonal entry
has chosen argument \(2\pi-\eps\), rather than \(\eps\).

The two possible values of \(\theta_{j\ell}\) can be written in
the single formula
$
\theta_{j\ell}
 =\eps(1-2c_{j\ell})+2\pi c_{j\ell}.
$
Using \eqref{eq:prefix-angles}, we therefore obtain
\[
\begin{aligned}
\pi+\sum_{j=2}^m\theta_{j\ell}
 =\pi+\eps\sum_{j=2}^m(1-2c_{j\ell})
   +2\pi\sum_{j=2}^m c_{j\ell}=\varphi_{m\ell}+2\pi d_\ell.
\end{aligned}
\]
Here the initial angle \(\pi\) comes from \(A_1=-\Id\), and
\(\varphi_{m\ell}\in(0,2\pi)\) is the chosen argument of the
\(\ell\)th diagonal entry of \(P_m\). Thus each occurrence of
\(c_{j\ell}=1\) contributes an extra \(2\pi\) to the sum of the
factor arguments.

Since all the matrices are diagonal, summing these differences
over \(\ell\) gives
\[
\begin{aligned}
2\pi K
 &=\sum_{j=1}^m\tr\Theta(A_j)-\tr\Theta(P_m)\\
 &=\sum_{\ell=1}^p
  (
   \pi+\sum_{j=2}^m\theta_{j\ell}-\varphi_{m\ell}
   )=2\pi\sum_{\ell=1}^p d_\ell,
\end{aligned}
\]
where \(K\) is the integer appearing in \eqref{eq:Delta-K}.
Moreover, by the choice of the entries \(c_{j\ell}\),
\[
K
 =\sum_{\ell=1}^p d_\ell
 =\sum_{j=2}^m\sum_{\ell=1}^p c_{j\ell}
 =\sum_{j=2}^m q_j.
\]

Equation~\eqref{eq:Delta-K} now determines \(\Delta_m\) at this
representative. Since \(\mathcal R_{\mathbf q}\) is path-connected
and \(\Delta_m\) is locally constant, the same value holds at every
point of \(\mathcal R_{\mathbf q}\). Consequently,
\begin{equation}\label{eq:Delta-on-piece}
\Delta_m=(m-1)p-2\sum_{j=2}^m q_j
\quad\text{on }\mathcal R_{\mathbf q}.
\end{equation}

\subsection{Joining the pieces with the same total index}

\begin{lemma}\label{lem:transfer}
Suppose $r\neq s$, $q_r<p$ and $q_s>0$. Let $\mathbf q'$ be obtained from $\mathbf q$ by replacing $q_r$ by $q_r+1$ and $q_s$ by $q_s-1$. Then $\mathcal R_{\mathbf q}$ and $\mathcal R_{\mathbf q'}$ lie in the same path component of $\X_{p,m}$.
\end{lemma}
\begin{proof}
We construct a path in \(\X_{p,m}\) from a diagonal representative
of \(\mathcal R_{\mathbf q}\) to a diagonal representative of
\(\mathcal R_{\mathbf q'}\).
Recall that the binary entries defining a representative in
\eqref{eq:binary-representatives} satisfy
$
c_{j\ell}\in\{0,1\},
$ $
\sum_{\ell=1}^p c_{j\ell}=q_j.
$
We may choose them so that
$
c_{r1}=0,$ $c_{s1}=1.
$
Indeed, \(q_r<p\) allows us to place all \(q_r\) ones in the
remaining \(p-1\) positions of row \(r\). Similarly, \(q_s>0\)
allows us to place one of the \(q_s\) ones in the first position
of row \(s\). Since \(r\neq s\), these two choices can be made
independently.

For this representative, the first diagonal arguments in factors
\(r\) and \(s\) are
$
\theta_{r1}(0)=\eps,$ $
\theta_{s1}(0)=2\pi-\eps.
$
We exchange these two arguments continuously by setting
\[
\theta_{r1}(t)=\eps+(2\pi-2\eps)t,\quad
\theta_{s1}(t)=2\pi-\eps-(2\pi-2\eps)t,
\quad 0\leq t\leq1.
\]
Let \(A_j(t)\) be the resulting diagonal matrices, with all other
entries kept fixed.

We verify that this path remains in \(\X_{p,m}\). First, both
varying arguments lie in
$
[\eps,2\pi-\eps]\subset(0,2\pi).
$
Hence neither of the corresponding eigenvalues equals \(1\).
All other eigenvalues are unchanged, so
\(A_j(t)\in\E_p\) for every \(j\) and \(t\).
Second, the sum of the two varying arguments is constant:
$
\theta_{r1}(t)+\theta_{s1}(t)=2\pi.
$
Thus
$
e^{i\theta_{r1}(t)}e^{i\theta_{s1}(t)}=1
$
throughout the path. Since all matrices remain diagonal, their
product is computed separately in each diagonal position.
The first diagonal entry of the total product is therefore
constant, and all its other diagonal entries are unchanged.
Consequently,
$
A_1(t)\cdots A_m(t)=A_1(0)\cdots A_m(0)\in\E_p.
$
Together with the factor conditions, this proves that the entire
path lies in \(\X_{p,m}\).

At \(t=1\), the two arguments have been exchanged:
$
\theta_{r1}(1)=2\pi-\eps,$ $
\theta_{s1}(1)=\eps.
$
The endpoint is therefore another binary representative, obtained
by replacing \(c_{r1}=0\) with \(1\) and \(c_{s1}=1\) with \(0\).
Its row sums are
$
q'_r=q_r+1,$ $q'_s=q_s-1$,
$q'_j=q_j$$(j\neq r,s)$.
By the preceding construction of binary representatives, this
endpoint belongs to \(\mathcal R_{\mathbf q'}\).

Finally, both \(\mathcal R_{\mathbf q}\) and
\(\mathcal R_{\mathbf q'}\) are path connected. An arbitrary point
of the first piece can be joined to the chosen initial
representative within that piece, and the terminal representative
can be joined to an arbitrary point of the second piece within
the second piece. Concatenating these paths with the path
constructed above proves that the two pieces lie in the same
path component of \(\X_{p,m}\).

The connecting path need not remain in \(\mathcal R_{p,m}\):
some intermediate products may acquire the eigenvalue \(1\).
This is allowed in \(\X_{p,m}\), which requires only the individual
factors and the total product to belong to \(\E_p\).
\end{proof}

If $\mathbf q$ and $\mathbf q'$ have the same sum, they are related by finitely many transfers of this form. To see this, whenever they differ choose $r$ with $q_r<q'_r$ and $s$ with $q_s>q'_s$. These inequalities imply $q_r<p$ and $q_s>0$. A transfer decreases $\sum_j|q_j-q'_j|$ by two. Repeating the procedure terminates at $\mathbf q'$. When there is only one coordinate, equality of sums already means equality of vectors. Thus all the pieces in \eqref{eq:Delta-on-piece} with a given value of $\Delta_m$ belong to one path component of the full space.

\subsection{Reaching the partial-product locus}

\begin{lemma}\label{lem:perturb}
Every point of $\X_{p,m}$ can be joined by a path in $\X_{p,m}$ to a point of $\mathcal R_{p,m}$.
\end{lemma}
\begin{proof}
Let $(A_1,\ldots,A_m)\in\X_{p,m}$, and write
$P_j=A_1\cdots A_j$. We will multiply $A_1$ by a small
scalar phase, which rotates every partial product by the same angle.
Since $A_1$ and $P_m$ have no eigenvalue equal to $1$, continuity
allows us to choose $0<\delta<2\pi$ such that $e^{iu}A_1,\ e^{iu}P_m\in\E_p$ for every $u\in[0,\delta]$.
Thus any rotation through an angle in this interval preserves
the required conditions on the first factor and the total product.

We next choose an angle that also works for all partial products
at the endpoint. For any eigenvalue $\lambda$ of any $P_j$,
the equation
$
 e^{i\tau}\lambda=1
$
has at most one solution in $(0,\delta)$, because $\delta<2\pi$.
Since $P_1,\ldots,P_m$ have only finitely many eigenvalues altogether,
only finitely many angles in $(0,\delta)$ are excluded.
Choose $\tau$ outside this finite set. Then $e^{i\tau}P_j\in\E_p$ for every $j=1,\ldots,m$.

Now define a path by
\[
 A_1(t)=e^{it\tau}A_1,\quad
 A_j(t)=A_j\quad(j\geq2),\quad 0\leq t\leq1.
\]
Its partial products satisfy
$
 P_j(t)=A_1(t)\cdots A_j(t)=e^{it\tau}P_j.
$
Because $0\leq t\tau<\delta$, our choice of $\delta$ ensures
that $A_1(t)$ and $P_m(t)$ remain in $\E_p$ throughout the path.
The other factors are unchanged and already belong to $\E_p$.
Hence the entire path lies in $\X_{p,m}$.

At $t=1$, our choice of $\tau$ ensures that every partial product
belongs to $\E_p$. Thus the endpoint lies in $\mathcal R_{p,m}$,
as required.
\end{proof}

\begin{proof}[Proof of Theorem~\ref{thm:products-intro}]
By Lemma~\ref{lem:perturb}, every point of $\X_{p,m}$ can be
joined to the partial-product locus $\mathcal R_{p,m}$ by a path
in $\X_{p,m}$. Since $\Delta_m$ is locally constant, its value
is unchanged along this path.

On $\mathcal R_{\mathbf q}$, formula~\eqref{eq:Delta-on-piece} gives
\[
 \Delta_m=(m-1)p-2\sum_{j=2}^m q_j.
\]
Every integer $K$ between $0$ and $(m-1)p$ can be written as a sum
of $m-1$ integers in $\{0,\ldots,p\}$. Thus the possible values
of $\Delta_m$ on $\mathcal R_{p,m}$, and hence on $\X_{p,m}$,
are exactly $(m-1)p-2K$ for $0\leq K\leq(m-1)p$.

Now take two points of $\X_{p,m}$ with the same value of $\Delta_m$,
and join them to points in $\mathcal R_{\mathbf q}$ and
$\mathcal R_{\mathbf q'}$, respectively. The displayed formula
implies that $\mathbf q$ and $\mathbf q'$ have the same sum.
Each piece is path connected, and successive applications of
Lemma~\ref{lem:transfer} connect the two pieces by paths in
$\X_{p,m}$. Concatenating these paths with the initial paths
connects the original two points. Hence every level set of
$\Delta_m$ is path connected.

Since $\Delta_m$ is continuous and has a discrete image, points
with different values cannot lie in the same connected component.
Therefore its level sets are precisely the connected components
of $\X_{p,m}$. The homotopy assertion for $m=2$ follows from
Proposition~\ref{prop:pair}.
\end{proof}

\begin{remark}
The proof provides a way to connect any two given points with the
same value of $\Delta_m$. However, the choices in this construction
are made separately for each pair of endpoints, and we have not
shown that they can be made continuously as the endpoints vary.
Thus these paths do not by themselves give a deformation retraction
of $\X_{p,m}$ onto its diagonal locus: such a retraction would require
a single continuous deformation defined on the whole space.
\end{remark}

\section{Surface group representations}\label{sec:surfaces}

\subsection{Genus zero}

If $g=0$ and $n=1$, the boundary is null-homotopic, so its holonomy is $\Id$ and $\Srep_{0,1}(p)$ is empty. If $n\geq2$, the relation determines the last boundary matrix:
$
 C_n=(C_1\cdots C_{n-1})^{-1}.
$
Using \eqref{eq:rho-inverse-sum}, we obtain a diffeomorphism
\begin{equation}\label{eq:planar-identification}
 \Srep_{0,n}(p)\cong\X_{p,n-1},\qquad
 \boldsymbol{\rho}=\Delta_{n-1}.
\end{equation}
Theorem~\ref{thm:products-intro} therefore gives
\[
 \boldsymbol{\rho}=(n-2)p-2K,\quad 0\leq K\leq(n-2)p,
 \qquad k=p+K.
\]
This proves the genus-zero case of Theorem~\ref{thm:main}, including the annulus: when $n=2$, the space is a copy of $\E_p$, and $k=p$.

\subsection{Connected commutator fibers}

For $g\geq1$, write
\[
 \mu_g:\U(p)^{2g}\longrightarrow\SU(p),\quad
 \mu_g(A_1,B_1,\ldots,A_g,B_g)=\prod_{r=1}^g[A_r,B_r].
\]

\begin{theorem}\label{thm:commutator}
The map $\mu_g$ is surjective, and every fiber is connected.
\end{theorem}
\begin{proof}
For a compact connected simply connected semisimple group $G$, surjectivity and connectedness of every fiber of the product-of-$g$-commutators map are stated in \cite[Fact~3]{HL2003}; the connectedness input is \cite[Theorem~7.2]{AMM1998}. Apply this theorem to $G=\SU(p)$ for $p\geq2$.

To pass to $\U(p)$, fix $Q\in\SU(p)$. Multiplication of each handle matrix by an arbitrary unit scalar defines a continuous map
\[
 (\mu_g^{\SU(p)})^{-1}(Q)\times\U(1)^{2g}
       \longrightarrow(\mu_g^{\U(p)})^{-1}(Q).
\]
It is surjective. Indeed, for each $U\in\U(p)$ choose $z\in\U(1)$ with $z^p=\det U$ and write $U=zV$ with $V\in\SU(p)$. Scalar factors disappear from commutators. These choices only establish pointwise surjectivity; no global choice of a $p$th root is asserted. The source is connected, so its image is connected. Surjectivity onto $\SU(p)$ follows already from the $\SU(p)$ theorem. For $p=1$, the target consists of one point and the fiber is the connected torus $\U(1)^{2g}$.
\end{proof}

For clarity, surjectivity alone also has a direct matrix proof. Write
\[
 Q=V\diag(q_1,\ldots,q_p)V^{-1},\qquad \prod_{\ell=1}^p q_\ell=1.
\]
Let $Se_\ell=e_{\ell+1}$ with cyclic indices, and choose unit scalars $d_\ell$ satisfying
$d_\ell/d_{\ell-1}=q_\ell$, where $d_0=d_p$. They exist because the product of the $q_\ell$ is one. For $D=\diag(d_1,\ldots,d_p)$,
\[
 [D,S]=\diag(q_1,\ldots,q_p).
\]
Conjugating by $V$ and setting the remaining handle pairs equal to the identity gives a preimage of $Q$. The connectedness of the entire fiber is the substantive classical theorem used above.

\subsection{Allowing the boundary matrices to vary}

\begin{lemma}\label{lem:closed-map}
Let $f:X\to Y$ be a continuous closed surjection. If $Y$ is connected and every fiber of $f$ is connected, then $X$ is connected.
\end{lemma}
\begin{proof}
Suppose $X=U\sqcup V$ is a separation into nonempty disjoint open sets. Each is also closed. A connected fiber cannot meet both, so $f(U)$ and $f(V)$ are disjoint. They are nonempty closed sets whose union is $Y$, since $f$ is closed and surjective. Each is the complement of the other and hence open, contradicting connectedness of $Y$.
\end{proof}

For an integer $k$ define the space of boundary data
\[
 \B_k=\{(C_1,\ldots,C_n)\in\E_p^n:
                \sum_{j=1}^n\tr\Theta(C_j)=2\pi k\}.
\]
In logarithmic coordinates this is
\begin{equation}\label{eq:convex-base}
\{(\Theta_1,\ldots,\Theta_n)\in\mathcal D_p^n:
                         \sum_{j=1}^n\tr\Theta_j=2\pi k\},
\end{equation}
an intersection of a convex set with an affine hyperplane. It is nonempty precisely when $1\leq k\leq np-1$. Necessity was proved in Lemma~\ref{lem:discrete}; for sufficiency take
\begin{equation}\label{eq:central-boundary}
 \Theta_1=\cdots=\Theta_n=\frac{2\pi k}{np}\Id.
\end{equation}
Thus every nonempty $\B_k$ is contractible.

Let $\Srep_k=\{\phi\in\Srep_{g,n}(p):k(\phi)=k\}$, with $g\geq1$. Restricting a representation to its boundary gives
\[
 \beta:\Srep_k\longrightarrow\B_k.
\]
The product $\prod_{i=1}^nC_i$ has determinant one on $\B_k$, and the fiber over $(C_1,\ldots,C_n)$ is
\begin{equation}\label{eq:boundary-fiber}
 \mu_g^{-1}((C_1\cdots C_n)^{-1}).
\end{equation}
Theorem~\ref{thm:commutator} says exactly that this fiber is nonempty and connected.

The map $\beta$ is also closed. In fact $\Srep_k$ is the closed subset
\[
 \{(\mathbf A,\mathbf B,\mathbf C)\in\U(p)^{2g}\times\B_k:
                      \mu_g(\mathbf A,\mathbf B)C_1\cdots C_n=\Id\}
\]
of a product with a compact first factor. Projection from such a product is closed.  Equivalently, $\beta$ is proper, since its preimage of a compact subset of $\B_k$ is closed in a compact product.

Lemma~\ref{lem:closed-map} now shows that $\Srep_k$ is connected. To conclude path connectedness, eliminate $C_n$ from the presentation. For $n\geq1$ this identifies the full representation space with
$\U(p)^{2g+n-1}$. The conditions that $C_1,\ldots,C_{n-1}$ and
\[
 C_n=\left(\prod_{r=1}^g[A_r,B_r]C_1\cdots C_{n-1}\right)^{-1}
\]
avoid $1$ are open conditions. Therefore $\Srep_{g,n}(p)$ is a smooth open submanifold of real dimension $(2g+n-1)p^2$, whenever it is nonempty. Its locally constant $k$-levels are open submanifolds, and hence locally path-connected. A connected locally path-connected space is path-connected. This proves the positive-genus part of Theorem~\ref{thm:main}.

\subsection{The quotient and the closed-surface case}

\begin{proof}[Completion of the proof of Theorem~\ref{thm:main}]
Let
$q:\Srep_{g,n}(p)\to\Srep_{g,n}(p)/\U(p)$
be the quotient map. Since $k$ is unchanged under unitary
conjugation, it defines a function on the quotient by
$\bar k([\phi])=k(\phi)$. This function is integer-valued
and continuous by the definition of the quotient topology.

Fix an integer $a$ for which $\bar k^{-1}(a)$ is nonempty.
Any two points of this level set have representatives in
$k^{-1}(a)$. By the preceding argument, $k^{-1}(a)$ is path
connected, so these representatives can be joined by a path
within that level set. Applying $q$ to this path gives a path
between the two quotient points. Thus every nonempty level set
of $\bar k$ is path connected.

On the other hand, a continuous integer-valued function is
constant on every connected set. Hence points with different
values of $\bar k$ cannot belong to the same connected component.
It follows that the nonempty level sets of $\bar k$ are precisely
the connected components of the quotient. The quotient map
therefore gives a natural bijection between the connected
components of the representation space and those of the quotient,
with corresponding components having the same value of $k$.

For completeness, if $n=0$, there is no boundary condition. For $g=0$ the representation space is a point. For $g\geq1$ it is $\mu_g^{-1}(\Id)$ and is connected by Theorem~\ref{thm:commutator}. Its quotient is connected as well. In particular the $n=0$ case has one component and should not be obtained by substituting $n=0$ into the positive-boundary formula.
\end{proof}

\section{Signatures and applications}\label{sec:applications}

\subsection{The Hermitian intersection form}

Let $\mathcal E_\phi$ be the rank-$p$ flat bundle associated with a unitary representation, with its parallel positive Hermitian form $\Omega$. We consider the first cohomology 
\[
 \widehat H^1(\Sigma;\mathcal E_\phi)
 =\im\bigl(H^1(\Sigma,\partial\Sigma;\mathcal E_\phi)
                 \longrightarrow H^1(\Sigma;\mathcal E_\phi)\bigr).
\]
Cup product, the pairing $\Omega$, and evaluation on the relative fundamental class give a nondegenerate skew-Hermitian form $Q$ on this image. With the convention of \cite[Section~2.1]{KPW2022}, the signature is the signature of the Hermitian form $iQ$. More explicitly, if $u$ is represented by a relative lift $\widetilde u$ and $v$ by an absolute class, the pairing is
\[
 Q(u,v)=\langle\Omega(\widetilde u\smile v),[\Sigma,\partial\Sigma]\rangle.
\]
This is independent of the relative lift on the image, by the long exact sequence and Poincar\'e--Lefschetz duality.

\begin{proposition}\label{prop:signature}
For $\phi\in\Srep_{g,n}(p)$ with $n\geq1$,
\begin{equation}\label{eq:signature}
 \sign(\phi)=\boldsymbol{\rho}(\phi)=np-2k(\phi).
\end{equation}
In particular, the signature is a complete invariant of the components of both $\Srep_{g,n}(p)$ and its conjugation quotient.
\end{proposition}
\begin{proof}
The compact specialization of the signature formula \cite[Theorem~4, Section~6.2, equation~(6.7)]{KPW2022} gives $T(\phi)=0$ and
\[
 \sign(\phi)=\sum_{j=1}^n\sum_{\ell=1}^p
                  \left(1-\frac{\theta_{j\ell}}\pi\right)
\]
when every $\theta_{j\ell}$ belongs to $(0,2\pi)$. This is \eqref{eq:signature}; completeness follows from Theorem~\ref{thm:main}.
\end{proof}

\begin{proposition}\label{prop:inertia}
For every representation in $\Srep_{g,n}(p)$, $n\geq1$, the natural map from relative to absolute first cohomology is an isomorphism and
\begin{align}
 \dim_\C H^1(\Sigma;\mathcal E_\phi)&=p(2g+n-2),\label{eq:cohom-dim}\\
 b_+(\phi)&=p(g+n-1)-k(\phi),\label{eq:bplus}\\
 b_-(\phi)&=p(g-1)+k(\phi).\label{eq:bminus}
\end{align}
Here $b_\pm$ are the positive and negative indices of $iQ$.
\end{proposition}
\begin{proof}
Give each boundary circle $c_j$ a cellular decomposition with
one vertex and one edge.Using the cellular local-coefficient construction
\cite[Section~3.H, pp.~328 and~334]{Hatcher2002},
choose lifts $\widetilde v$ and $\widetilde e$ of the vertex
and the oriented edge, with
$\partial\widetilde e=(t-1)\widetilde v$.
Since the generator $t$ acts on the coefficient space
$\mathbb C^p$ by $C_j$, the induced coboundary map is
$v\mapsto(C_j-\Id)v$.
Thus the cohomology is computed by the two-term complex
$\mathbb C^p\xrightarrow{C_j-\Id}\mathbb C^p$
in degrees zero and one.
Thus $H^0(c_j;\mathcal E_\phi)=\ker(C_j-\Id)$ and
$H^1(c_j;\mathcal E_\phi)=\operatorname{coker}(C_j-\Id)$.
Since $1$ is not an eigenvalue of $C_j$, the map
$C_j-\Id$ is invertible, so both groups vanish.
There are no cochains in higher degrees.
As $\partial\Sigma$ is a finite disjoint union of these circles,
we conclude that $H^*(\partial\Sigma;\mathcal E_\phi)=0$.

The long exact sequence of the pair $(\Sigma,\partial\Sigma)$
therefore gives
\[
 0\longrightarrow H^1(\Sigma,\partial\Sigma;\mathcal E_\phi)
 \longrightarrow H^1(\Sigma;\mathcal E_\phi)
 \longrightarrow 0.
\]
Thus the natural map from relative to absolute first cohomology
is an isomorphism.

We next compute their common dimension. A global flat section
is determined by its value at a basepoint, and this value must
be fixed by every holonomy. In particular, it must lie in
$\ker(C_1-\Id)=0$. Hence $H^0(\Sigma;\mathcal E_\phi)=0$.
Since $\Sigma$ is compact with nonempty boundary, it deformation
retracts onto a finite graph $G$; see
\cite[Example~1B.2, p.~88]{Hatcher2002}.
Cohomology with local coefficients is homotopy invariant and
can be computed using cellular cochains
\cite[Section~3.H, p.~334]{Hatcher2002}.
Since $G$ has no cells of dimension two or higher, we obtain
\[
H^2(\Sigma;\mathcal E_\phi)
\cong H^2(G;\mathcal E_\phi|_G)=0.
\]

For a rank-$p$ local system, each cell contributes a copy of
$\C^p$ to the corresponding cellular cochain group. The alternating
sum of the cochain dimensions is therefore $p\chi(\Sigma)$.
This equals the alternating sum of the cohomology dimensions.
Since only the first cohomology can be nonzero, we obtain
\[
 \dim_{\C}H^1(\Sigma;\mathcal E_\phi)
 =-p\chi(\Sigma)
 =p(2g+n-2).
\]
Together with the relative-to-absolute isomorphism above,
this proves \eqref{eq:cohom-dim}.

Finally, $b_+$ and $b_-$ count the positive and negative
eigenvalues of the Hermitian form $iQ$. Since this form is
nondegenerate, there are no zero eigenvalues, so their sum
equals the dimension of the cohomology space. Their difference
is the signature, which is $np-2k$. Thus
\[
 b_++b_-=p(2g+n-2),\qquad
 b_+-b_-=np-2k.
\]
Adding the equations and dividing by two gives $b_+$;
subtracting the second from the first and dividing by two gives
$b_-$. Hence
\[
 b_+=p(g+n-1)-k,\qquad
 b_-=p(g-1)+k,
\]
which proves \eqref{eq:bplus}--\eqref{eq:bminus}.
\end{proof}

For $g=0$, nonnegativity of the two indices recovers $p\leq k\leq(n-1)p$. When $n\geq3$, the component $k=p$ has positive definite intersection form, and the component $k=(n-1)p$ has negative definite intersection form. These are statements about the twisted cohomological form; no irreducibility or rigidity of the representation is implied. In positive genus, when the space is nonempty, the possible signatures are
\[
 \{np-2,np-4,\ldots,2-np\}.
\]
The parity restriction comes from boundary acyclicity, or directly from \eqref{eq:signature}. It distinguishes this locus from the full unitary representation space considered in \cite{KPW2026}, where eigenvalue $1$ is allowed and intermediate signatures of both parities can occur.

\subsection{Simple representatives of every component}

\begin{proposition}\label{prop:normal-forms}
Assume $n\geq1$. Every component in genus zero contains a representation with diagonal image. Every component in positive genus contains a representation with all boundary holonomies equal to a single scalar matrix and all but one handle pair trivial.
\end{proposition}
\begin{proof}
In genus zero, let $p\leq k\leq(n-1)p$. Choose integers $d_1,\ldots,d_p\in\{1,\ldots,n-1\}$ with $\sum_\ell d_\ell=k$. Such a choice exists by distributing $k-p$ units among $p$ entries, each initially equal to one and each having capacity $n-2$. Set
\[
 C_1=\cdots=C_n=
 \diag(e^{2\pi i d_1/n},\ldots,e^{2\pi i d_p/n}).
\]
Each entry avoids $1$, the product of the $n$ matrices is $\Id$, and the index is $k$. By Theorem~\ref{thm:main}, every representation in this component can be joined to this one.

In positive genus, put $\lambda=e^{2\pi i k/(np)}$ and prescribe $C_j=\lambda\Id$ for every $j$. Their product is $e^{2\pi i k/p}\Id$, whose determinant is one. The explicit commutator construction after Theorem~\ref{thm:commutator} realizes its inverse using a single handle. Take all other handle matrices to be $\Id$. Again Theorem~\ref{thm:main} provides the deformation from any point in the component to this representative.
\end{proof}

The second assertion does not imply that every connected component
contains a representation with abelian image. For example, on a once-punctured surface every representation with abelian image has trivial boundary holonomy, so no such representation belongs to $\Srep_{g,1}(p)$. The noncommuting handle matrices in Proposition~\ref{prop:normal-forms} are necessary in that case.

\subsection{Irreducible representations and the smooth quotient}

A unitary representation is called irreducible if it preserves no nonzero proper complex subspace of $\C^p$. Write $\Srep_{g,n}^{\mathrm{irr}}(p)$ for the irreducible locus.

\begin{proposition}\label{prop:irreducible}
Assume $p\geq2$, $n\geq1$, and $d=2g+n-1\geq2$. Every component of $\Srep_{g,n}(p)$ contains irreducible representations, and its irreducible locus is path-connected. Thus $k$ gives the same component classification for $\Srep_{g,n}^{\mathrm{irr}}(p)$ and its conjugation quotient. The latter is a smooth manifold of real dimension
\[
 (d-1)p^2+1=(2g+n-2)p^2+1.
\]
\end{proposition}
\begin{proof}
After eliminating the final boundary generator, a representation
is determined by $d=2g+n-1$ unitary matrices. Thus the ambient
representation space is $M=\U(p)^d$, of real dimension $dp^2$.
The boundary conditions define an open subset
$\Srep_{g,n}(p)\subset M$.

We first estimate the dimension of the reducible locus.
A tuple $(A_1,\ldots,A_d)$ is reducible precisely when all its
matrices preserve a common subspace
$V\subset\C^p$ of dimension $\ell$, for some
$1\leq\ell\leq p-1$. Since the matrices are unitary, they also
preserve $V^\perp$. Relative to the decomposition
$\C^p=V\oplus V^\perp$, each matrix is therefore block diagonal,
with blocks in $\U(\ell)$ and $\U(p-\ell)$.

For a fixed $\ell$, let $\mathcal I_\ell$ be the space of pairs
consisting of an $\ell$-plane $V$ and a tuple preserving $V$.
The choice of $V$ is parametrized by $\Gr_\ell(\C^p)$.
Once $V$ is fixed, the possible tuples form
$(\U(\ell)\times\U(p-\ell))^d$.
Thus $\mathcal I_\ell$ is a compact smooth fiber bundle, and
\[
 \dim_{\mathbb R}\mathcal I_\ell
 =2\ell(p-\ell)+d\bigl(\ell^2+(p-\ell)^2\bigr)
 =dp^2-2(d-1)\ell(p-\ell).
\]

Let $Z_\ell\subset M$ be the image obtained by forgetting $V$.
It consists of the tuples having a common invariant
$\ell$-plane. Since $\mathcal I_\ell$ is compact, $Z_\ell$ is
compact and hence closed.

We also need that $Z_\ell$ is semialgebraic. An $\ell$-plane
can be represented by its orthogonal projection $P$, which
satisfies
$
 P^2=P,$ $P^*=P$, $\tr P=\ell.
$
A unitary matrix $A_j$ preserves this plane if and only if
$PA_j=A_jP$. These conditions are polynomial equations in
the real and imaginary parts of the matrix entries.
The Tarski--Seidenberg theorem therefore shows that forgetting
$P$ gives a semialgebraic set $Z_\ell$; see \cite{BCR}.
Moreover, a semialgebraic projection cannot increase dimension,
so $\dim_{\mathbb R}Z_\ell\leq\dim_{\mathbb R}\mathcal I_\ell$.

The full reducible locus is the finite union
$Z=\bigcup_{\ell=1}^{p-1}Z_\ell$.
Since $d\geq2$ and $\ell(p-\ell)\geq1$, the dimension estimate
above gives
$
 \dim_{\mathbb R}Z\leq dp^2-2.
$
Thus $Z$ is closed and semialgebraic, with real codimension
at least two in $M$.

Now let $W$ be a connected component of $\Srep_{g,n}(p)$.
It is an open connected manifold, and hence is path connected.
Since $Z$ has dimension strictly smaller than $\dim W=dp^2$,
it cannot contain the open set $W$. Therefore
$W\setminus Z$ is nonempty.

We show that any two points of $W\setminus Z$ can be joined
without passing through $Z$. Choose a path between them in $W$,
and approximate it by a smooth path that is constant near its
endpoints. Since $Z$ is semialgebraic, it admits a finite smooth
Whitney stratification; see \cite{BCR}. In particular, it is a
finite union of smooth submanifolds, called strata, each of
codimension at least two in $M$.

By relative transversality, we can perturb the path arbitrarily
slightly, keeping it fixed near its endpoints, so that it is
transverse to every stratum; see \cite[Chapter~3]{Hirsch}.
The perturbed path can be kept inside $W$, because the original
path has compact image contained in this open set.

Such a transverse path cannot meet any stratum. Indeed, if
$\gamma(t)$ belonged to a stratum $S$, transversality would require
$
 d\gamma_t(T_t[0,1])+T_{\gamma(t)}S
 =T_{\gamma(t)}M.
$
The tangent space of the path contributes at most one dimension,
whereas $S$ has codimension at least two. The left-hand side
therefore has dimension at most $\dim M-1$, a contradiction.
Hence the perturbed path avoids $Z$ entirely. This proves that
$W\setminus Z$ is path connected.

It remains to describe the quotient of the irreducible locus.
By Schur's lemma, a unitary matrix commuting with an irreducible
tuple must be scalar. Thus the stabilizer of every irreducible
tuple in $\U(p)$ is exactly the scalar subgroup $\U(1)$.
After dividing out this subgroup, the induced action of
$\mathrm{PU}(p)=\U(p)/\U(1)$ is free. It is also proper, since
$\mathrm{PU}(p)$ is compact.

The irreducible locus $\Srep_{g,n}^{\mathrm{irr}}(p)$ is open
in $M$, because $Z$ is closed. By the quotient manifold theorem
\cite[Theorem~21.10]{Lee2013}, the smooth, free and proper
action of $\mathrm{PU}(p)$ on
$\Srep_{g,n}^{\mathrm{irr}}(p)$ gives a smooth quotient
of real dimension 
\[
 dp^2-(p^2-1)
 =(d-1)p^2+1
 =(2g+n-2)p^2+1.
\]
The $\mathrm{PU}(p)$-orbits are the same as the $\U(p)$-orbits,
so this is the desired conjugation quotient.

Finally, each nonempty $k$-level in the full representation
space is a component $W$. We have proved that its irreducible
part $W\setminus Z$ is nonempty and path connected.
Its image in the quotient is therefore also path connected.
Since $k$ descends to a continuous integer-valued function,
different values cannot lie in the same connected component.
Thus $k$ gives the asserted component classification for both
the irreducible locus and its quotient.
\end{proof}

The assumption $d=2g+n-1\geq2$ is essential here: the fundamental
group must have at least two free generators. For an annulus,
the fundamental group is $\mathbb Z$, so a representation is
determined by a single unitary matrix. Every eigenline of this
matrix is invariant under the entire representation. Thus, when
$p>1$, every representation of the annulus is reducible.

When $p=1$, on the other hand, every representation is irreducible,
since $\mathbb C$ has no nonzero proper subspaces.

For a sphere with $n\geq3$ punctures, the fundamental group is free
of rank $n-1\geq2$. Hence, when $p\geq2$,
Proposition~\ref{prop:irreducible} applies and shows that every
nonempty connected component of $\Srep_{0,n}(p)$ contains
irreducible representations. 

\subsection{Prescribed boundary conjugacy classes}

Suppose $g\geq1$ and $n\geq1$, and fix conjugacy classes
$\mathcal C_1,\ldots,\mathcal C_n\subset\U(p)$.
Let $\mathcal R$ be the space of representations
$\phi:\pi_1(\Sigma_{g,n})\to\U(p)$ satisfying
$\phi(c_j)\in\mathcal C_j$ for every boundary generator $c_j$.
Since conjugate matrices have the same determinant, each class
$\mathcal C_j$ has a well-defined determinant
$\delta_j\in\U(1)$.

We claim that $\mathcal R$ is nonempty if and only if
$\delta_1\cdots\delta_n=1$, and that it is connected whenever
it is nonempty.

To see the necessity of the determinant condition, write
$A_i=\phi(a_i)$, $B_i=\phi(b_i)$, and $C_j=\phi(c_j)$.
Every commutator has determinant one. Taking determinants
therefore gives $\delta_1\cdots\delta_n=1$.

Conversely, suppose this condition holds, and choose any
$C_j\in\mathcal C_j$. Then $(C_1\cdots C_n)^{-1}$ belongs to
$\mathrm{SU}(p)$. The commutator map
\[
 \mu_g:\U(p)^{2g}\longrightarrow\mathrm{SU}(p),
 \quad
 (A_1,B_1,\ldots,A_g,B_g)
 \longmapsto [A_1,B_1]\cdots[A_g,B_g]
\]
is surjective by Theorem~\ref{thm:commutator}.
Hence we can choose the matrices $A_i,B_i$ so that the surface
relation holds. This proves that $\mathcal R$ is nonempty.

We now prove connectedness. Consider the map that records only
the boundary matrices:
\[
 \pi:\mathcal R\longrightarrow
 \mathcal C_1\times\cdots\times\mathcal C_n,
 \qquad
 \phi\longmapsto
 \bigl(\phi(c_1),\ldots,\phi(c_n)\bigr).
\]
The preceding argument shows that $\pi$ is surjective.
For a fixed boundary tuple $(C_1,\ldots,C_n)$, the remaining
choices are precisely the matrices $A_i,B_i$ satisfying
$\mu_g(A_1,B_1,\ldots,A_g,B_g)=(C_1\cdots C_n)^{-1}$.
Thus the fiber of $\pi$ is naturally identified with
$
 \mu_g^{-1}\bigl((C_1\cdots C_n)^{-1}\bigr),
$
which is connected by Theorem~\ref{thm:commutator}.

Each conjugacy class $\mathcal C_j$ is compact and connected,
since it is the continuous image of $\U(p)$ under conjugation
of a fixed matrix. Their product is therefore compact and
connected. Moreover, $\mathcal R$ is a closed subset of the
compact space
$\U(p)^{2g}\times\mathcal C_1\times\cdots\times\mathcal C_n$,
so it is compact. Consequently, $\pi$ is a closed map.
Lemma~\ref{lem:closed-map} now applies: a closed continuous
surjection with connected fibers onto a connected space has
a connected domain. Hence $\mathcal R$ is connected.

This is a special case of the more general theory of
representations with prescribed boundary conjugacy classes
in \cite{HL2005}. Here the boundary matrices are allowed to
have eigenvalue $1$.

\section{A second route in genus zero}\label{sec:classical}

We give another proof of the genus-zero classification using
classical results on the eigenvalues of products of unitary
matrices. The argument has three steps: we show that the possible
boundary spectra form a convex set, that the representations
with any fixed boundary spectra form a connected space, and that
the spectral map is closed. Together, these facts imply
connectedness of each nonempty level of the invariant $k$.

Assume $n\geq3$. In genus zero, a representation is determined
by its boundary matrices $C_1,\ldots,C_n\in\U(p)$, which satisfy
$C_1\cdots C_n=\Id$. Since no $C_j$ has eigenvalue $1$, its
eigenvalues can be written uniquely, in increasing order of
their normalized arguments, as
\[
 e^{2\pi i\alpha_{j1}},\ldots,e^{2\pi i\alpha_{jp}},
 \qquad
 0<\alpha_{j1}\leq\cdots\leq\alpha_{jp}<1.
\]
Taking determinants in the product relation shows that
$\sum_{j=1}^n\sum_{\ell=1}^p\alpha_{j\ell}$ is an integer.
We fix this integer and denote it by $k$.

Let $\mathcal P_k$ be the set of eigenangle arrays
$(\alpha_{j\ell})$ that occur for such representations with
total sum $k$. Thus $\mathcal P_k$ records the possible boundary
spectra, without recording the matrices themselves.
The spectral map is the continuous surjection
\[
 \sigma:\Srep_{0,n}(p)_k\longrightarrow\mathcal P_k,
 \qquad
 (C_1,\ldots,C_n)\longmapsto(\alpha_{j\ell}).
\]

We first show that $\mathcal P_k$ is convex.
The multiplicative eigenvalue theorem, in the index formulation
of \cite[Theorem~2.1]{FW2006}, characterizes the realizable
arrays by finitely many affine inequalities in the eigenangles.
The ordering conditions and the equation
$\sum_{j,\ell}\alpha_{j\ell}=k$ are also linear.
Consequently, a line segment joining two points of
$\mathcal P_k$ remains in $\mathcal P_k$.
The strict endpoint conditions $0<\alpha_{j\ell}<1$ are preserved
along such a segment, and repeated eigenvalues are allowed.
In particular, $\mathcal P_k$ is connected whenever it is
nonempty. Restricting the angles to $(0,1)$ also avoids the
ambiguity that angles $0$ and $1$ represent the same eigenvalue.

We next show that every fiber of $\sigma$ is connected.
Fix an eigenangle array in $\mathcal P_k$.
For unitary matrices, fixing all eigenvalues, including their
multiplicities, is equivalent to fixing a conjugacy class.
Thus the fiber consists of tuples in prescribed conjugacy
classes whose product is $\Id$.

To apply the connectedness theorem of
Alekseev--Malkin--Meinrenken, we pass from $\U(p)$ to $\SU(p)$.
Set $s_j=\sum_{\ell=1}^p\alpha_{j\ell}$ and define
$
 D_j=e^{-2\pi i s_j/p}C_j.
$
Then
$\det D_j=e^{-2\pi i s_j}\det C_j=1$, so $D_j\in\SU(p)$.
For the fixed eigenangle array, this scalar multiplication
identifies the prescribed conjugacy class of $C_j$ with a
conjugacy class $\mathcal D_j$ in $\SU(p)$.
Here conjugacy under $\U(p)$ agrees with conjugacy under
$\SU(p)$: a unitary conjugating matrix can be multiplied by a
scalar to make its determinant one, without changing its
conjugation action.

Since $\sum_j s_j=k$, the original product relation becomes
\[
 D_1\cdots D_n=e^{-2\pi i k/p}\Id.
\]
The matrix on the right belongs to the center of $\SU(p)$,
because $k$ is an integer.
Hence the fiber of $\sigma$ is homeomorphic to a fiber of
the multiplication map
\[
 m:\mathcal D_1\times\cdots\times\mathcal D_n
 \longrightarrow\SU(p),
 \qquad
 (D_1,\ldots,D_n)\longmapsto D_1\cdots D_n.
\]

For $p\geq2$, the group $\SU(p)$ is compact, connected, and
simply connected. Each $\mathcal D_j$ is a compact connected
conjugacy class. The fusion construction equips their product
with a quasi-Hamiltonian structure whose group-valued moment
map is precisely the multiplication map $m$.
By \cite[Theorem~7.2]{AMM1998}, every fiber of $m$ is connected.
This proves connectedness of the fibers of $\sigma$.
When $p=1$, fixing the eigenangles fixes all the boundary
matrices, so each nonempty fiber consists of a single point.

It remains to show that $\sigma$ is closed.
We do this by proving that it is proper, meaning that the
inverse image of every compact subset of $\mathcal P_k$
is compact.
Let $K\subset\mathcal P_k$ be compact.
Since all its eigenangles lie strictly between $0$ and $1$,
there is an $\varepsilon>0$ such that
\[
 \varepsilon\leq\alpha_{j\ell}\leq1-\varepsilon
 \qquad
 \text{for every }(\alpha_{j\ell})\in K.
\]
Now take a sequence in $\sigma^{-1}(K)$.
Compactness of $\U(p)^n$ gives a subsequence converging to a
unitary tuple $(C_1,\ldots,C_n)$.
The product relation passes to the limit.
The uniform bound on the eigenangles ensures that no limiting
matrix acquires eigenvalue $1$.
Moreover, the ordered eigenangles vary continuously on this
region, so the limiting eigenangle array still belongs to $K$.
Thus the limit lies in $\sigma^{-1}(K)$, proving compactness.

Both the domain and the target are locally compact metrizable
spaces, so properness implies that $\sigma$ is closed.
We have therefore obtained a closed continuous surjection
onto a connected space, with connected fibers.
Lemma~\ref{lem:closed-map} shows that
$\Srep_{0,n}(p)_k$ is connected whenever it is nonempty.
By the local path connectedness established in
Section~\ref{sec:surfaces}, this level is also path connected.

To determine which levels are nonempty, we use the index
bounds in \cite[Theorem~2.2]{FW2006}.
In our setting, no boundary matrix has eigenvalue $1$.
There is also no trivial summand: a nonzero vector in such a
summand would be fixed by every boundary matrix.
The correction terms in those bounds therefore vanish, giving
$
 p\leq k\leq(n-1)p.
$
The explicit diagonal construction in the proof of
Proposition~\ref{prop:normal-forms} realizes every integer in this interval;
this construction does not use Theorem~\ref{thm:main}.
Since $k$ is continuous and integer-valued, different levels
lie in different connected components.
Thus the nonempty levels are exactly the connected components,
and their number is $(n-2)p+1$.
The annulus is treated separately by
\eqref{eq:planar-identification}.

This gives a proof of the genus-zero classification that uses
established convexity and connectedness theorems, independently
of the explicit paths in Section~\ref{sec:products}.
The matrix proof provides further geometric information:
it describes how the pieces defined by partial products are
joined by paths, including paths along which an intermediate
product has eigenvalue $1$.
It also gives the Grassmannian model in the two-factor case.

There is also an algebro-geometric interpretation.
After choosing a complex structure on the punctured sphere,
the Mehta--Seshadri correspondence relates unitary
representations, up to conjugacy, to polystable parabolic
bundles of parabolic degree zero \cite{MS1980}.
The boundary spectra determine the parabolic weights,
with the precise identification depending on the convention
for local monodromy.
This correspondence underlies the approach to spectral
inequalities in \cite{Belkale2001}.
Our proofs work directly with unitary matrices and allow
the boundary eigenangles to vary throughout $(0,1)$.

\bibliographystyle{amsalpha}
\bibliography{references}
\end{document}